\documentclass[11pt]{article}
\usepackage{mathtools}
\usepackage[margin=1.5in]{geometry} 
\usepackage{latexsym, paralist, xypic, stmaryrd, soul}
\usepackage{amsmath,amsthm,amssymb,amsfonts}
\usepackage{outlines}
\usepackage{enumitem}
\usepackage[utf8]{inputenc}
\usepackage[english]{babel}
\usepackage{mathrsfs}
\usepackage{tikz-cd}
\usepackage[all]{xy}
\usepackage{booktabs} 
\usepackage{xcolor}
\usepackage{multicol}
\usepackage[colorlinks]{hyperref}
\usepackage[center]{titlesec}

\definecolor{capri}{rgb}{0.2, 0.45, 1.0}

\hypersetup{citecolor=capri, linkcolor=capri}

\DeclareMathOperator*{\limone}{lim^{\rlap{\scriptsize 1}}}

\newcommand{\N}{\mathbb{N}}

\theoremstyle{remark}
\newtheorem{remark}{Remark}[section]

\theoremstyle{plain}
\newtheorem{theorem}{Theorem}[section]
\newtheorem{lemma}[theorem]{Lemma}

\newtheorem{proposition}[theorem]{Proposition}

\newtheorem{definition-theorem}[theorem]{Definition / Theorem}

\theoremstyle{definition}
\newtheorem{definition}[theorem]{Definition}

\newcommand{\addresses}{
		\bigskip
		\footnotesize
		
		A.~Jaime, \textsc{Department of Mathematics, University of Hawai'i at M{\=a}noa,
			Honolulu, Hawai'i 96822}\par\nopagebreak
		\textit{E-mail address}: \texttt{ajaime@hawaii.edu}
		
}

\title{Topology of $KK$-theory via inverse limits of discrete abelian groups}
\author{Arturo Jaime}
\date{}

\begin{document}

\maketitle

\begin{abstract}
	This paper seeks to characterize some topological properties of pro-countable topological abelian groups. Using the Milnor exact sequence given by the controlled picture of $KK$-theory by Willett and Yu, we describe some topological properties of the topological group $KK(A,B)$ with respect to the satisfaction of Mittag-Leffler and stability conditions of certain inverse systems.
\end{abstract}

\vspace{5mm}

\tableofcontents

\vspace{10mm}

\section{Introduction}

Generalizing both $K$-theory and $K$-homology of $C^{*}$-algebras, the Kasparov groups, $KK_{*}(A,B)$, for separable $C^{*}$-algebras $A$ and $B$, are an important tool in study of $C^{*}$-algebras and noncommutative geometry \cite{Kas95}. One aspect of the Kasparov groups is that there is a canonical (possibly non-Hausdorff) topological structure on $KK(A,B)$ making it into a topological group. Investigations into this topological structure has its origins in the work of Brown-Douglas-Fillmore \cite{BDF}. Further investigations of this topology characterizing the closure of zero with quasidiagonal elements were done in $\cite{Brown80, Salinas}$. In  a series of papers \cite{Schoc, SchocII} Schochet introduced a topology on $KK(A,B)$ showing that it was compatible with respect to $KK(A,B)$'s algebraic structure for a nice class of $C^{*}$-algebras. In $\cite{Dad}$, Dadarlat unified varying descriptions of topologies on $KK$-theory, showing that indeed they were all equivalent. In that same paper, Dadarlat further studied the generalized R{\o}rdam group $KL(A,B) := KK(A,B)/ \overline{\{0\}}$ with respect to stably approximately unitarily equivalence of $*$-homomorphisms between $C^{*}$-algebras, showing that $KL(A,B)$ is a Polish group.

Recently \cite{WilletYu25}, Willett and Yu introduced controlled $KK$-theory groups that are related to $KK(A,B)$ via a Milnor exact sequence for any pair of separable $C^{*}$-algebras $A$ and $B$. The relevant aspect being that the inverse limit of the controlled $KK$-theory groups were identified with the R{\o}rdam group $KL(A,B)$ and the closure of $0 \in KK(A,B)$ identified with an appropriate lim$^{1}$ group with the indiscrete topology. The key point being that for any pair of separable $C^{*}$-algebras $A$ and $B$, the R{\o}rdam group is a pro-countable abelian group. 

We can show that as a topological space $KK(A,B)$ is homeomorphic to the product space $KL(A,B) \times \overline{\{0\}}$. Using the fact that $KL(A,B)$ is pro-countable we seek to characterize the possible topological structures on $KK(A,B)$ since topological pro-countable abelian groups, i.e., inverse limits of systems of discrete countable abelian groups, are straightforward to understand. For instance the compactness of $KL(A,B)$ and the vanishing of $\overline{\{0\}}$ are related to stability and Mittag-Leffler conditions being satisfied by certain inverse systems of controlled $KK$-theory groups. The main theorem being Theorem \ref{Topology}. We also show in Proposition \ref{Ror} show that the generalized R{\o}rdam group $KL(A,B)$ is totally disconnected answering a question of Dadarlat asked in \cite{Dad}

\subsection*{Acknowledgments}

This contents of this paper were part of the author's thesis. I would like to thank my advisor Rufus Willett for his support. The author was partially supported by US NSF grant DMS 2247968.

\vspace{5mm}

\section{Inverse Limits of Topological Groups}

We begin by stating and proving some basic theorems of inverse limits of topological groups for the convenience of the reader. For reference see Chapter 3 of \cite{Bourbaki}.

\begin{lemma}\label{group basis}
	Let $G$ be an abelian topological group and $\mathcal{B}$ a basis neighborhood for $0 \in G$. Then for any $g \in G$, the set $$\big\{ g + V \;|\; V \in \mathcal{B} \big\}$$ is a basis neighborhood around $g$.
\end{lemma}

We state and prove a standard result of topological groups below that is used for the main result.

\begin{proposition}\label{top grp iso}
	Let $G$ be an abelian topological group. Then $$G \cong G\big /\,\overline{\{0\}} \times \overline{\{0\}}$$ as topological spaces.
\end{proposition}

\begin{proof}
	Let $s: G /\,\overline{\{0\}} \rightarrow G$ be any set-theoretic section and let $\pi: G \rightarrow G/\overline{\{0\}}$ denote the quotient map. Then define a map $$f: G\big /\,\overline{\{0\}} \times \overline{\{0\}} \rightarrow G, \;\text{ by }\; (x, h) \rightarrow s(x) + h.$$ The map is surjective, to see this let $g \in G$ be arbitrary, then consider $g - s(\overline{g}) \in \overline{\{0\}}$, and $$f(\overline{g}, g - s(\overline{g})) = g.$$ \\ For injectivity, assume that $f(x_{1}, h_{1}) = f(x_{2}, h_{2})$, that is $s(x_{1}) + h_{1} = s(x_{2}) + h_{2}$. Then \begin{align}\label{top grp iso eq}
		s(x_{1}) = s(x_{2}) + h_{2} - h_{1},
	\end{align} this implies that $$x_{1} = \pi\big(s(x_{1})\big) = \pi\big(s(x_{2}) + h_{2} - h_{1}\big) = \pi\big(s(x_{2})\big) = x_{2}.$$ Thus $s(x_{1}) = s(x_{2})$, which along with equation 9(\ref{top grp iso eq}) implies that $h_{1} = h_{2}$. Thus $f$ is injective.
	
	To show that $f$ is continuous, it suffices to consider an arbitrary basis neighborhood, $g + V \subset G$, where $U \in \mathcal{B}$ and show that $f^{-1}(g + V)$ is open. Since $g + V$ is an open neighborhood, and the quotient map $\pi$ is an open map, then $\pi(g + V)$ is open in $G/\overline{\{0\}}$. If $(y,h') \in \pi(g + V) \times \overline{\{0\}}$, then $$f(y,h') = s(y) + h' = g + v + h',$$  for some $v \in V$ since $y \in \pi(g + V)$. Now by Lemma \ref{group basis}, $V$ is an open basis neighborhood of $0$, so it follows that $\overline{\{0\}} \subset V$. Since $V$ is an open neighborhood of $v$, there exists a basis neighborhood such that $v \in v + U \subset V$, where $U$ is a basis neighborhood of $0$, i.e. $U \in \mathcal{B}$. Since $\overline{\{0\}} \subset W$, then $v + z \in  V$ for any $z \in \overline{\{0\}}$. Thus $v + h' \in V$, whence $f(y,h') = g + v + h' \in g + V$. Thus, the open set $\pi(g + V) \times \overline{\{0\}}$ is contained in $f^{-1}(g+V)$. In the other direction, let $(x,h) \in f^{-1}(g + V)$ be arbitrary. Then $s(x) + h = g + v$ for some $v \in V$, and $$s(x) = g + v + (-h), $$ however $v + (-h) \in V$ by above. So $s(x) \in g + V$, thus $x \in \pi(g + V)$ and so \begin{align}\label{top grp subset eq}
		f^{-1}(g+V) \subset \pi(g + V) \times \overline{\{0\}} \subset f^{-1}(g+V).
	\end{align}	
	An inverse map to $f$ is given by sending $g \in G$ to $(\pi(g), g - s(\pi(g))) \in G/\overline{\{0\}}$. That this map is continuous is given by the inclusion (\ref{top grp subset eq}) above.
\end{proof}

\begin{definition}
	Let $\Lambda$ be a  directed set and for each $\alpha \in \Lambda$, let $G_{\alpha}$ be a discrete abelian group, and $\beta \in \Lambda$ such that $\alpha \leq \beta$, there is a homomorphism $f_{\alpha\beta}$ mapping $G_{\beta}$ to $G_{\alpha}$; moreover, if  there exists $\gamma \in \Lambda$ such that $\alpha, \beta \leq \gamma$ then we have the following commutative diagram: \[\begin{tikzcd}
		& G_{\beta} \arrow{dr}{f_{\alpha\beta}} \\
		G_{\gamma} \arrow{ur}{f_{\beta\gamma}} \arrow{rr}{f_{\alpha\gamma}} & & G_{\alpha}& 
	\end{tikzcd} \] Then $(G_{\alpha}, f_{\alpha \beta})_{\alpha, \beta \in \Lambda}$ is called an \textit{inverse system} of topological  groups. \end{definition}
	
	We will only consider inverse systems of countable discrete abelian groups.
	
	\begin{definition}
	Consider the topological group $\prod_{\alpha \in \Lambda}G_{\alpha}$ with the product topology and let $\pi_{\alpha}$ denote the canonical projection sending $x \in \prod_{\alpha \in \Lambda}G_{\alpha}$ to $x_{\alpha} \in G_{\alpha}$. The \textit{inverse limit} of $(G_{\alpha}, f_{\alpha \beta})_{\alpha, \beta \in \Lambda}$ is defined as the subset of coherent tuples $$\varprojlim G_{\alpha} := \Big\{ (x_{\alpha}) \in \prod_{\alpha \in \Lambda}G_{\alpha}\,\, \Big|\,\, \pi_{\alpha}(x) = f_{\alpha\beta}(\pi_{\beta}(x))  \text{ whenever } \alpha \leq \beta \Big\}, $$ endowed with subspace topology. It will be notationally convenient to denote the restriction of $\pi_{\alpha}$ to $\varprojlim G_{\alpha}$ as simply $\pi_{\alpha}$. \end{definition}

\begin{remark}\label{basis intersection}
	A general basis element for the topology on $\varprojlim G_{\alpha}$ defined via the product topology is a finite intersection of sets of the form $\pi^{-1}_{\alpha}(S_{\alpha})$ where $S_{\alpha} \subset G_{\alpha}$ is any subset, so of the form $$\bigcap_{i=1}^{n}\pi^{-1}_{\alpha_{i}}(S_{\alpha_{i}}).$$ Since $\Lambda$ is a directed set there exists $\alpha_{0} \in \Lambda$ such that $\alpha_{i} \leq \alpha_{0}$ for $1 \leq i \leq n$. Define $$G_{\alpha_{0}}  \supseteq S := \bigcap_{i=1}^{n}f^{-1}_{\alpha_{i} \alpha_{0}}(S_{\alpha_{i}}).$$ We claim that $\pi^{-1}_{\alpha_{0}}(S) = \bigcap_{i=1}^{n}\pi^{-1}_{\alpha_{i}}(S_{\alpha_{i}})$. To see this let $x \in \pi^{-1}_{\alpha_{0}}(S)$, then $x_{\alpha_{0}} \in S$ and so $x_{\alpha_{i}} = f_{\alpha_{i}\alpha_{0}}(x_{\alpha_{0}}) \in S_{i}$ for $1 \leq i \leq n$, thus $x \in \bigcap_{i=1}^{n}\pi^{-1}_{\alpha_{i}}(S_{i})$. Conversely, let $x \in \bigcap_{i=1}^{n}\pi^{-1}_{\alpha_{i}}(S_{\alpha_{i}})$, then $x_{\alpha_{i}} \in S_{\alpha_{i}}$ for $1 \leq i \leq n$, thus $x_{\alpha_0} \in S$ and so $x\in \pi^{-1}_{\alpha_{0}}(S)$.
	
	Thus a general basis element for the topology on $\varprojlim G_{\alpha}$ is of the form $\pi^{-1}_{\alpha}(S)$ where $S \subset G_{\alpha}$ is any subset.
\end{remark} 

The following is a standard result of inverse limits (see \cite{Bourbaki}).

\begin{proposition}\label{pr1}
	Let $I$ be a cofinal subset of $\Lambda$, that is for every $\alpha \in \Lambda$ there exists a $\beta \in I$ such that $\alpha \leq \beta$. Then as topological groups we have the following isomorphism $$\varprojlim G_{\alpha} \cong \varprojlim G_{\beta}.$$
\end{proposition}

\begin{proof}
	Define the map $\Phi: \prod_{\alpha \in \Lambda} G_{\alpha} \rightarrow \prod_{\beta \in I} G_{\beta}$ to be the map restricting a tuple $x = (x_{\alpha})_{\alpha \in \Lambda}$ to the its coordinates index by the subset $I \subset \Lambda$, that is $x \mapsto (x_{\beta})_{\beta \in I}$. The map $\Phi$ descends down to a map $$\Phi: \varprojlim G_{\alpha} \rightarrow \varprojlim G_{\beta} \;\text{ defined by }\; \varprojlim G_{\alpha} \ni x \mapsto \big(\pi_{\beta}(x)\big)_{\beta \in I} = (x_{\beta})_{\beta \in I} \in \varprojlim G_{\beta}.$$ Again we note that a general basis element for the topology on $\varprojlim G_{\beta}$ is of the form $\pi^{-1}_{\beta}(S)$ where $S \subset G_{\beta}$ is any subset. Since $\pi_{\beta} (\Phi(x)) = \pi_{\beta}(x)$ for $\beta \in I$, $\Phi$ is continuous. It is straightforward to check that $\Phi$ is a homomorphism. 
	
	To show injectivity, let $x, y \in \varprojlim G_{\alpha}$ with $\Phi(x) = \Phi(y)$, then let $\alpha \in \Lambda$. Since $I$ is cofinal, there exists $\beta \in I$ such that $\alpha \leq \beta$. So there exists a homomorphism $\pi_{\alpha\beta}: G_{\beta} \rightarrow G_{\alpha}$ and $$x_{\beta} = \pi_{\beta}(\Phi(x)) = \pi_{\beta}(\Phi(y)) = y_{\beta},$$ thus $x_{\alpha} =\pi_{\alpha\beta}(x_{\beta}) = \pi_{\alpha\beta}(y_{\beta}) = y_{\alpha}$. Since $\alpha$ was arbitrary we have that $x = y$, and so $\Phi$ is injective. 
	
	For surjectivity we basically construct an inverse map. Let $(x_{\beta})_{\beta \in I} \in \varprojlim G_{\beta}$, then for any $\alpha \in \Lambda$, there exists a $\beta \in I$ (by cofinality) such that $\alpha \leq \beta$. Then set $x_{\alpha} := f_{\alpha\beta}(x_{\beta})$. This choice is unique, indeed let $\gamma \in I$ be such that $\alpha \leq \gamma$ as well. Since $I$ is a directed set there exists $\delta \in I$ such that $\beta, \gamma \leq \delta$. Then $$f_{\alpha\beta}(x_{\beta}) = f_{\alpha\beta}f_{\beta\delta}(x_{\delta}) = f_{\alpha\delta}(x_{\delta}) = f_{\alpha\gamma}f_{\gamma\delta}(x_{\delta}) = f_{\alpha\gamma}(x_{\gamma})$$ Thus the assignment $x_{\alpha} := f_{\alpha\beta}(x_{\beta})$ is well defined and defines an element that mapped to $(x_{\beta})_{\beta \in I}$ by $\Phi$. Thus $\Phi$ is an isomorphism of topological groups. 
\end{proof}

\begin{remark}
	Moving forward we will assume that our general inverse system $(G_{\alpha}, f_{\alpha, \beta})_{\alpha, \beta \in \Lambda}$ admits a countable cofinal subsystem. With the isomorphism above we will work with the countably indexed inverse system $(G_{n}, f_{n, m})_{n,m\in \mathbb{N}}$. To simplify notation we write $f_{n+1}$ for $f_{n, n+1}$.
\end{remark}

\vspace{5mm}

\section{Metrizability}

The main goal of this section is to prove some standard results of inverse limits of discrete groups such metrizability, completeness, and separability in order to conclude that such groups are Polish groups.

\begin{definition}
	A topological space is said to be a \textit{Polish space} if it is separable, complete, and metrizable. A topological group that is a Polish space is a \textit{Polish group}.
\end{definition}

\begin{lemma}
	Let $(G_{n}, f_{n, m})_{n,m\in \mathbb{N}}$ be an inverse system of discrete abelian groups. Then the space $\varprojlim G_{n}$ is metrizable.
\end{lemma}

\begin{proof}
	Without loss of generality we add $G_{0} := \{e\}$ into the inverse system, with $G_{1} \rightarrow G_{0}$ being the trivial map.
	
	Let $x, y \in \varprojlim G_{n}$, then define $d(x,y) := 2^{-m}$, where $m = \min\{n > 0 \;| \; x_{n} \neq y_{n} \}$, and if no such $n$ exists, then $x = y$ and we set $d(x,y) = 0$. Positive definiteness and symmetry follow from construction, we need only verify the triangle-inequality. 
	
	Consider $x,y, z \in \varprojlim G_{n}$ with $x \neq z$, else $d(x,z) = 0 \leq d(x,y) + d(y,z)$. Define $m_{1} = \min_{n}\{x_{n} \neq y_{n} \}$, $m_{2} = \min_{n}\{y_{n} \neq z_{n} \}$, and $m_{3} = \min_{n}\{x_{n} \neq z_{n} \}$. Without loss of generality assume that $m_{1} \leq m_{2}$. There are two cases.
	
	Assume that $m_{1} = m_{2}$, then it follows that $x_{m_{1} - 1} = y_{m_{1}-1} = z_{m_{1} -1 }$, thus $m_{1} \leq m_{3}$, and $$d(x,z) = \frac{1}{2^{m_{3}}} < \frac{1}{2^{m_{1}}} + \frac{1}{2^{m_{1}}} = d(x,y) + d(y,z).$$
	
	On the other hand, let $m_{1} < m_{2}$. It follows that $y_{m_{1}} = z_{m_{1}}$, and $x_{m_{1} - 1} = y_{m_{1}-1} = z_{m_{1} -1 }$. Thus $m_{1} \leq m_{3}$ from which it follows that $$d(x,z) = \frac{1}{2^{m_{3}}} \leq \frac{1}{2^{m_{1}}} + \frac{1}{2^{m_{2 }}} = d(x,y) + d(y,z).$$ 
	
	We next show that the topology induced by this metric agrees with the inverse limit topology. Consider a subset of $\varprojlim G_{n}$ denoted by $B(x, \frac{1}{2^n})$ for some $n \in \mathbb{N}$, and of the form $$B\Big(x, \frac{1}{2^n}\Big) := \pi^{-1}_{n-1}\big(\{x_{n-1}\}\big).$$ Such sets form a basis for the metric topology on $\varprojlim G_{n}$. On the other hand a basis element for the inverse limit topology on $\varprojlim G_{n}$ is of the form $\pi^{-1}_{n}(S)$ where $S \subset G_{n}$ is any subset. In order to get a better hand on the elements of $\pi^{-1}_{n}(S)$, we replace $S$ with $S'$ where $$S' = \pi_{n}\big(\pi_{n}^{-1}(S)\big)$$ and we note that $\pi^{-1}_{n}(S) = \pi^{-1}_{n}(S')$ but with the added property that for each $g \in S'$ there exists $y \in \varprojlim G_{n}$ such that $y_{n} = g$. The reason for needing to do this is because the inverse system is not necessarily an inverse system comprising surjective maps. For each $g \in S'$ choose $\tilde{g} \in \projlim G_{n}$ with $\pi_{n}(\tilde{g}) = g$. Then we may write  $\pi^{-1}_{n}(S)$ as a union of open balls $$\pi^{-1}_{n}(S) = \pi^{-1}_{n}(S') = \bigcup_{g \in S'} B\big(\tilde{g}^{i}, 2^{-(n+1)}\big).$$ Thus the topology induced by the metric is equivalent to the  inverse limit topology.
\end{proof}

\begin{lemma}
	Let $(G_{n}, f_{n, m})_{n,m\in \mathbb{N}}$ be an inverse system of discrete abelian groups. The space $\varprojlim G_{n}$ is complete.
\end{lemma}

\begin{proof}
	Let $(x^{i})$ be a Cauchy sequence in $\varprojlim G_{n}$. For each $n \in \mathbb{N}$ there exists a $k_{n} \geq 1$ such that for all $i,j \geq k_{n}$ $d(x^{i}, x^{j}) < 2^{-n}$, that is $x^{i}_{n} = x^{j}_{n}$. We define $x \in \varprojlim G_{n}$ by $x_{n} = x^{i}_{n}$ where $i \geq k_{n}$. Let $\varepsilon > 0$, then there exists $n \in \mathbb{N}$ such that $2^{-n} \leq \varepsilon$ and hence $k_{n} \in \mathbb{N}$ such that $d(x^{i}, x^{j}) < 2^{-n}$ for all $i,j \geq k_{n}$. Thus for all $i \geq k_{n}$, $d(x, x^{i}) \leq 2^{-n} \leq \varepsilon$. Thus $x^{i} \rightarrow x$ as $i \rightarrow \infty$ and so $\varprojlim G_{n}$ is complete.
\end{proof} 

\begin{lemma}
	Let $(G_{\alpha}, f_{\alpha \beta})_{\alpha, \beta \in \Lambda}$ be an inverse system of discrete countable abelian groups. The space $\varprojlim G_{n}$ is separable.
\end{lemma}

\begin{proof}
	With each $G_{j}$ countable, we may enumerate $$(g_{j,i})^{n_{j}}_{i=1} = \pi_{j}(\varprojlim G_{n}) \subset G_{j}, \;\text{where }\, n_{j} \in \mathbb{N}\cup \{\infty\}.$$ We again consider such a restriction as our inverse system may not necessarily comprise surjective maps. We then make a choice of $y \in \pi^{-1}_{j}(\{g_{j,i}\})$ for each $i$ and let $F_{j} \subset \pi^{-1}_{j}(G_{j})$ be the countable set of all such chosen $y$. Thus for each $g \in \pi_{j}(\varprojlim G_{n})$, there exists $y \in F_{j}$ such that $\pi_{j}(y) = g$. Now define $$F := \bigcup_{j=1}^{\infty} F_{j}.$$ Since $F$ is a countable union of countable sets it follows that $F$ is countable. 
	
	Now let $U$ be any nonempty open subset of  $\varprojlim G_{n}$, then $U$ contains some nonempty basis element $\pi_{n}^{-1}(S)$ of the topology for some subset $S \subset G_{n}$. Let $x \in \pi_{n}^{-1}(S)$ and consider $x_{n} \in G_{n}$. By construction there exists $y \in F_{n}$ such that $y_{n} = x_{n}$ and thus $y \in \pi_{n}^{-1}(\{x_{n}\}) \subset \pi_{n}^{-1}(S)$. Hence $U \cap F \neq \varnothing$, thus $\varprojlim G_{n}$ is separable.\\
\end{proof}

\begin{proposition}
	The topological group $\varprojlim G_{n}$ is a Polish group. \hspace*{\fill} $\Box$
\end{proposition}

\vspace{5mm}

\section{The long exact lim - lim$^{1}$ sequence}

The main goal of this section is to recall some results involving the Mittag-Leffler condition and the lim - lim$^{1}$ exact sequence of inverse limits. 

Elements of the inverse limit $\varprojlim G_{n}$ are coherent sequences of $\prod_{n=1}^{\infty} G_{n}$, so may be expressed as \[\begin{tikzcd}
	&\cdot\cdot\cdot \arrow{r}{f_{3,4}} &x_{3} \arrow{r}{f_{2,3}} &x_{2} \arrow{r}{f_{1,2}} &x_{1}. 
\end{tikzcd}\] Restricting each $G_{n}$ to the subgroup $f_{n,n+1}(G_{n+1}) \subset G_{n}$ does not change the inverse limit, in fact we can make reductions to the inverse system so as to produce an inverse system with surjective homomorphisms:

To obtain an inverse system with surjective maps, restrict each $G_{n}$ to what is being mapped to and discard everything else, that is define \begin{align*}
	G_{n}^{\circ} := \bigcap_{m\, \geq \, n} \text{Im}\big(f_{n,m}\big).
\end{align*}  It remains that $G_{n}^{\circ}$ is a topological group with the discrete topology and we must show this indeed defines an inverse system which with the same inverse limit, that is we have the following isomorphism of topological groups \begin{align*}
	\varprojlim G_{n}^{\circ} \cong \varprojlim G_{n}. 
\end{align*}

To prove this we will need some definitions and results first. An alternative definition for the inverse limit of an inverse system of abelian groups is the following.

\begin{definition}
	Let $(G_{\alpha}, f_{\alpha \beta})_{\alpha, \beta \in \Lambda}$ be a countable inverse system of topological abelian groups. Then we may define a map \begin{align*}
		\Delta: \prod_{n =1}^{\infty} G_{n} \rightarrow \prod_{n =1}^{\infty} G_{n}, \, \text{ by } \, (x_{n}) \mapsto (x_{n} - f_{n+1}(x_{n+1})).
	\end{align*} Then we may define the inverse limit as $\varprojlim G_{n} := \text{ker}(\Delta)$, and we also define the \textit{derived limit} $\varprojlim^{1} G_{n} := \text{coker}(\Delta)$. 
\end{definition}

\begin{definition}\cite{Atiyah61}
	An inverse system $(G_{n})$ is said to satisfy the \textit{Mittag-Leffler} condition if for each $n$ there exists an $m \geq n$ such that $$\text{Im}(f_{n,m}: G_{m}\rightarrow G_{n}) = \text{Im}(f_{n,k}: G_{k}\rightarrow G_{n})$$ for all $k \geq n$.
\end{definition}

The following connects the Mittag-Leffler condition with the vanishing of lim$^{1}$ (\cite{Atiyah61}, \cite{Gray}) \footnote{The statement is true in more generality for inverse systems of countable groups, see \cite{Geo}.}

\begin{proposition}\label{ML}
	Let $(G_{n})$ be an inverse system of abelian groups. Then $(G_{n})$ satisfies the Mittag-Leffler condition if and only if $\varprojlim^{1} G_{n} = 0$.
\end{proposition}

The following is a well-known result of inverse systems of abelian groups, (Chapter IX in \cite{BousKan}, Proposition 2.12 in \cite{HughesRanicki}).

\begin{proposition}
	Let $0 \rightarrow (K_{n}) \rightarrow (G_{n}) \rightarrow (H_{n}) \rightarrow 0$ be a short exact sequence of inverse sequences of abelian groups. Then there is a long-exact sequence of abelian groups$$
	0 \rightarrow \lim_{\leftarrow} K_{n} \rightarrow \lim_{\leftarrow} G_{n} \rightarrow \lim_{\leftarrow} H_{n} \rightarrow \lim_{\leftarrow}{}\!^1 K_{n} \rightarrow \lim_{\leftarrow}{}\!^1 G_{n} \rightarrow \lim_{\leftarrow}{}\!^1 H_{n} \rightarrow 0.
	$$
\end{proposition}

The following result shows that although the construction of $(G_{n}^{\circ})$ preserves the inverse limit it does not preserve the derived limit.

\begin{proposition}\label{surj rest}
	Let $(G_{n}^{\circ}), (G_{n})$ be inverse sequences as defined above, then $\varprojlim G_{n}^{\circ} \cong \varprojlim G_{n}$ and $0 \cong \varprojlim^{1} G_{n}^{\circ}$ may not be isomorphic to $\varprojlim^{1} G_{n}$.
\end{proposition}

\begin{proof}
	Observe that each homomorphism $f_{n}: G_{n} \rightarrow G_{n-1}$ restricts to a homomorphism $f_{n}|_{G_{n}^{\circ}}: G_{n}^{\circ} \rightarrow G_{n-1}^{\circ}$. Indeed, let $g \in G_{n}^{\circ}$, then for every $k \geq n$, $g \in \text{Im}(f_{n,k})$, so there exists some $y \in G_{k}$ such that $f_{n-1, n}(g) = f_{n-1,n}(f_{n,k}(y)) = f_{n-1,k}(y) \in \text{Im}(f_{n-1,k})$. With $k \geq n$ arbitrary it follows that $f_{n}(g) \in G_{n}^{\circ}$. Thus $(G_{n}^{\circ})$ forms an inverse system. Moreover, we then have the following short exact sequence of inverse systems $$0 \rightarrow G_{n}^{\circ} \rightarrow G_{n} \rightarrow G_{n}/G_{n}^{\circ} \rightarrow 0.$$ Thus we have a long exact sequence $$0 \rightarrow \varprojlim G_{n}^{\circ} \rightarrow \varprojlim G_{n} \rightarrow \varprojlim G_{n}/G_{n}^{\circ} \rightarrow \varprojlim{}^{1}  G_{n}^{\circ} \rightarrow \varprojlim{}^{1} G_{n} \rightarrow \varprojlim{}^{1} G_{n}/G_{n}^{\circ} \rightarrow 0.$$
	
	Observe that $(G_{n}^{\circ})$ satisfies the Mittag-Leffler condition, so $\varprojlim^{1}  G_{n}^{\circ} = 0$ and $\varprojlim^{1} G_{n} \cong \varprojlim{}^{1} G_{n}/G_{n}^{\circ}$.  Moreover, by construction, it follows that $\varprojlim G_{n}/G_{n}^{\circ} = 0$, thus from the long exact sequence it follows that $\varprojlim G_{n}^{\circ} \cong \varprojlim G_{n}$ as groups. Finally the bijective continuous map $\varprojlim G_{n}^{\circ} \rightarrow \varprojlim G_{n}$ is a closed map between Hausdorff spaces thus the map is a homeomorphism of topological groups.
\end{proof}

\vspace{5mm}

\section{Topology}
In this section we characterize topological properties of $\varprojlim G_{n}$, such as compactness and connectedness, by restricting our investigations to only the topological structure of the inverse limit as opposed to its structure as a topological group, see Proposition \ref{invtop1}. We do this by considering a homeomorphic description of $\varprojlim G_{n}$ given by the inverse system of surjective maps constructed in the section above. \\

\begin{remark} With proposition \ref{surj rest} in mind we note that from this point forward we will proceed to only consider a general inverse system with surjective homomorphisms, $(G_{n}, f_{n, m})$. Moreover, we only care to preserve the topology of $\varprojlim G_{n}$ instead of its structure as a topological group.\\
\end{remark} 

We first recharacterize the inverse limit with respect to the kernels of each homomorphism in the inverse system. 

\begin{remark}\label{totdisc}
Consider the surjective homomorphism $f_{1,2}: G_{2} \rightarrow G_{1}$ and denote the following topological spaces by $X_{1}:= G_{1}$ and $X_{2} := \text{ker}(f_{1,2})$, then as topological spaces we have $G_{2} \cong X_{1} \times X_{2}$. To see this, it suffices to construct a bijection since we have the discrete topology on both spaces. For each $a \in X_{1}$ we make a choice of some $x_{a} \in f_{1,2}^{-1}(a)$. With this choice consider the map sending $X_{1}\times X_{2} \ni (a,b) \mapsto x_{a} + b \in G_{2}$. To show surjectivity consider $c \in G_{2}$ and $a = f_{1,2}(c)$, then $(a, c - x_{a}) \mapsto c$. For injectivity, assume that for $(a,b),(c,d) \in X_{1}\times X_{2}$ that $x_{a} + b = x_{c} + d$. Then \begin{align*}
	a = f_{1,2}(x_{a}) = f_{1,2}(x_{a} + b) = f_{1,2}(x_{c} + d) = f_{1,2}(x_{c}) = c.
\end{align*} Thus $x_{a} = x_{c}$ and this in turn implies that $b = d$. To rephrase, we have the following commutative diagram in the category of sets \[\begin{tikzcd}
	G_{2}\arrow{r}{f_{1,2}}\arrow{d}{\cong} & G_{1}\arrow{d}{\cong}  \\
	X_{1}\times X_{2} \arrow{r}{\tau_{1}} & X_{1}
\end{tikzcd} \] where $\tau_{i}$ denotes the projection onto the first $i$ coordinates. Continuing inductively in this manner and defining $X_{n} := \text{ker}(f_{n-1, n})$, it follows that each $G_{n}$ in the inverse system can be written as $G_{n} \cong X_{1} \times \cdots \times X_{n}$. This allows us to rewrite the inverse limit, while preserving the topological structure, as \[\begin{tikzcd}
	& \cdots \rightarrow X_{1}\times X_{2} \times X_{3} \times X_{4} \arrow{r}{\tau_{3}} &X_{1} \times X_{2} \times X_{3} \arrow{r}{\tau_{2}} &X_{1}\times X_{2} \arrow{r}{\tau_{1}} &X_{1} &{}
\end{tikzcd} \] where each map in the inverse system is the map dropping the last coordinate. As inverse limits of topological spaces we have $$\varprojlim G_{n} \cong \varprojlim X_{1} \times \cdots \times X_{n}.$$ 
\end{remark}

With thus characterization of the inverse limit, we may begin to classify the topologies on $\varprojlim G_{n}$ via possible cardinalities of the kernels, $X_{n}$. To do so we need the following definition. \\

\begin{definition}
	Let $(A_{n}, f_{n-1,n})$ be an inverse system of topological spaces. The inverse limit $\varprojlim A_{n}$ is said to \textit{stabilize} if there exists some $n \in \mathbb{N}$ such that for all $m \geq n$, $f_{m,m+1}$ is a homeomorphism.
\end{definition}

For an inverse limit that stabilizes, after some point each $f_{m, m+1}$ is a homeomorphism so no new information is being added and thus ``collapsing" the inverse limit up to this point should not change its topology. In the kernel characterization, if $\varprojlim X_{1}\times \cdots \times X_{n}$ stabilizes, then after some point each kernel is equal to $\{0\}$, hence there exists some $n \in \mathbb{N}$ such that for $m \geq n$, $X_{m} \cong \{0\}$.

\begin{lemma}\label{lma2}
	Let $\varprojlim X_{1} \times 
	\cdots \times X_{n}$ be as above. If $\varprojlim X_{1} \times 
	\cdots \times X_{n}$ stabilizes then there exists a $k \in \mathbb{N}$ such that $$\varprojlim X_{1} \times 
	\cdots \times X_{n} \cong X_{1} \times 
	\cdots \times X_{k}.$$
\end{lemma}

\begin{proof}
	Choose $k \in \mathbb{N}$ such that for all $n \geq k$, $f_{n+1,n}$ is an isomorphism. Note that subsequence $\{n+k, n+k+1, n+k+2, ...\} \subset \mathbb{N}$ is a cofinal subsequence, thus by \ref{pr1} above we can restrict to the subsequence without changing the topology of the inverse limit, thus we have \begin{align*}
		\varprojlim_{n}  X_{1} \times 
		\cdots \times X_{n} \cong \varprojlim_{n} X_{1} \times 
		\cdots \times X_{n + k}.
	\end{align*} Then right hand side is indeed homeomorphic to $ X_{1} \times 
	\cdots \times X_{k}$ as each space in the inverse system is $ X_{1} \times 
	\cdots \times X_{k}$ and each map a homeomorphism.
\end{proof}

There are two possible cases: \begin{align*}
	\textbf{Case I}& \,- \,\varprojlim G_{n}\, \text{ stabilizes,} \\
	\textbf{Case II}& \,- \text{ Otherwise.}
\end{align*} In case I there are two further sub-cases, say cases I.1 and I.2., where either each $X_{n}$ is finite or there exists some infinite kernel. In case I.1 where each $X_{n}$ is finite and we stabilize, then by lemma \ref{lma2} our inverse limit is a finite space and thus compact. 

\begin{lemma}[Case I.1]\label{lma2b}
	Let  $\varprojlim X_{1} \times 
	\cdots \times X_{n}$ be a stabilizing inverse limit as above with all kernels of finite cardinality, then $\varprojlim X_{1} \times 
	\cdots \times X_{n}$ is homeomorphic to a finite set $X_{1} \times 
	\cdots \times X_{k}$ and thus compact. 
\end{lemma}

In case I.2 where there exist some kernels of infinite cardinality we claim that the inverse limit is locally compact.

\begin{lemma}[Case I.2]\label{lma3}
	Let  $\varprojlim X_{1} \times 
	\cdots \times X_{n}$ be a stabilizing inverse limit as above with finitely many kernels of infinite cardinality, then $\varprojlim X_{1} \times 
	\cdots \times X_{n} \cong \N$, and thus locally compact. 
\end{lemma}

\begin{proof}
	By lemma \ref{lma2} the inverse limit is homeomorphic to $ X_{1} \times 
	\cdots \times X_{n}$ for some $n \in N$. This space will be infinite discrete and thus locally compact. Since each $X_{i}$ is countable, thus $ X_{1} \times 
	\cdots \times X_{n}$ is countable as well. This implies the existence of a bijective map $f: \mathbb{N} \rightarrow X_{1} \times \cdots \times X_{n}$. Since both the domain and codomain have the discrete topology, it follows that $f$ and $f^{-1}$ are continuous. Thus $ X_{1} \times 
	\cdots \times X_{n} \cong \mathbb{N}$.
\end{proof}

Before proceeding to the next case we prove the following result which applies to all inverse limits of discrete abelian groups.

\begin{lemma}\label{totdisc lemma}
	Let  $\varprojlim X_{1} \times 
	\cdots \times X_{n}$ be as defined in remark \ref{totdisc}, then $\varprojlim X_{1} \times 
	\cdots \times X_{n}$ is totally disconnected.
\end{lemma}

\begin{proof}
	For total disconnectedness, let $x, y \in \varprojlim X_{1} \times \cdots \times X_{n}$ be distinct points. Then there exists some $k \in \mathbb{N}$ such that $x_{k} \neq y_{k} \in X_{k}$ then define $$U := X_{k} \backslash \{x_{k}\} \times X_{k-1} \times \cdots \times X_{1} \;\;\text{ and }\;\; V := \{x_{k}\} \times X_{k-1} \times \cdots \times X_{1}.$$ Note that each $X_{n}$ and $X_{1} \times \cdots \times X_{n}$ have the discrete topology thus both $U$ and $V$ are open. Moreover, $U \cap V = \varnothing$ and $U \cup V = X_{k} \times X_{k-1} \times \cdots \times X_{1}$, thus $U, V$ are clopen. Thus $\pi^{-1}_{k}(U)$ is clopen and by construction $y \in \pi^{-1}_{k}(U)$ and $x \notin \pi^{-1}_{k}(U)$. Thus we have separated $x$ and $y$ by a clopen set and thus $\varprojlim X_{1} \times \cdots \times X_{n}$ is totally disconnected. 
\end{proof}

Now for Case II, where the inverse system does not stabilize, we have three sub-cases. Our first sub-case, Case II.1 is when each $X_{n}$ is finite and thus so is each $G_{n}$. The product topology of $\prod_{n=1}^{\infty} G_{n}$ is compact by Tychonoff's theorem and since $\varprojlim G_{n}$ is a closed subspace it follows that $\varprojlim G_{n}$ is compact.

\begin{lemma}[Case II.1]\label{lma4}
	Assume that $\varprojlim X_{1} \times 
	\cdots \times X_{n}$ does not stabilize and that each $X_{n}$ is finite. Then $\varprojlim X_{1} \times 
	\cdots \times X_{n}$ is compact, metrizable, perfect, and totally disconnected. Thus $\varprojlim X_{1} \times 
	\cdots \times X_{n}$ is homeomorphic the Cantor set, $\mathcal{C}$.
\end{lemma}

\begin{proof}
	It follows from above that $\varprojlim X_{1} \times 
	\cdots \times X_{n}$ is compact and metrizability follows from Lemma 1.2. To show perfectness, we must show that for any $x \in \varprojlim X_{1} \times \cdots \times X_{n}$ and any basis neighborhood $U$ of $x$, there exists $y\in U$ such that $x \neq y$. Now $U$ is of the form $\pi_{j}^{-1}(S)$ for some $j \in \mathbb{N}$ where $\pi_{j}: \varprojlim X_{1} \times 
	\cdots \times X_{n} \longrightarrow X_{1} \times 
	\cdots \times X_{j}$ sending $x \mapsto \{x_{1}, ..., x_{j}\}$ and $S$ is any subset of $X_{1} \times 
	\cdots \times X_{j}$. Since the inverse limit does not stabilize, there exists some $k > j$ such that $X_{k} = \ker(f_{k-1,k})$ is non-trivial or $X_{k}\backslash \{x_{k}\} \neq \varnothing $. Define $y = (y_{n}) \in \varprojlim X_{1} \times \cdots \times X_{n}$ by $y_{n} = x_{n}$ for all $n \neq k$, and $y_{k}$ is some choice of element in $X_{k}\backslash \{x_{k}\}$. Then $y \in \pi^{-1}_{j}(S) = U$ and $x \neq y$. Thus $x$ is not an isolated point. Since this is true for any $x$, we have that $\varprojlim X_{1} \times \cdots \times X_{n}$ is perfect.
	
	Total disconnectedness follows from Lemma \ref{totdisc lemma}. Thus $\varprojlim X_{1} \times 
	\cdots \times X_{n}$ is homeomorphic the Cantor set.
\end{proof}

The next sub-case, Case II.2, is when there are finitely many non-finite $X_{n}$.

\begin{lemma}[Case II.2]\label{lma5}
	Assume that $\varprojlim X_{1} \times 
	\cdots \times X_{n}$ does not stabilize and that there exist only finitely many infinite $X_{n}$. Then $\varprojlim X_{1} \times \cdots \times X_{n} \cong \mathbb{N} \times \mathcal{C}$, thus is locally compact.
\end{lemma}

\begin{proof}
	Let $X_{k}$ denote the last occurring infinite kernel. Then as before we may reduce our inverse limit to the subsequence that begins at $X_{1} \times \cdots \times X_{k}$ and denote $X_{0} := X_{1} \times \cdots \times X_{k}$. It is the straightforward to check that \begin{align*}
		X_{0} \times \varprojlim (X_{k+1} \times \cdots \times X_{n}) \cong \varprojlim X_{1} \times 
		\cdots \times X_{n + k}.
	\end{align*} Note that by lemma \ref{lma4} we have that $\varprojlim (X_{k+1} \times \cdots \times X_{n})$ is compact.  Moreover $X_{0}$ is countable, infinite, and discrete and thus by Lemma 2.4 is homeomorphic to $\mathbb{N}$. Thus by Lemma 2.3 above we have that \begin{align*}
		X_{0} \times \varprojlim (X_{k+1} \times \cdots \times X_{n}) \cong \mathbb{N} \times \mathcal{C}.
	\end{align*} Since we are then taking the product space of a compact space with and infinite discrete space, the resulting space is locally compact.
\end{proof}

Finally, for the last case, Case II.3, where there exist infinitely many infinite kernels, we show that $\varprojlim X_{1} \times \cdots \times X_{n}$ is not locally compact.

\begin{lemma}[Case II.3]\label{lma6}
	Assume that $\varprojlim X_{1} \times 
	\cdots \times X_{n}$ does not stabilize and that there exist infinitely many non-finite $X_{n}$. Then $\varprojlim X_{1} \times \cdots \times X_{n}$ is not locally compact and homeomorphic $\mathbb{N}^{\mathbb{N}}$.
\end{lemma}

\begin{proof}
	Let $x$ be any element in $\varprojlim X_{1} \times 
	\cdots \times X_{n}$, and consider a basis neighborhood of $x$. Any basis neighborhood of $x$ is of the form $\pi^{-1}_{n}(S)$ where $S$ is any subset of $G_{n}$ and will contain the closed neighborhood $\pi^{-1}_{n}(\{x_{n}\}) \subset \pi^{-1}_{n}(S)$. However, observe that $\pi^{-1}_{n}(\{x_{n}\})$ can be written out as $$ \cdot\cdot\cdot \times X_{n+2} \times X_{n+1}\times \{x_{n}\} \times \{x_{n-1}\} \times \cdot\cdot\cdot \times \{x_{1}\}.$$ With infinitely many non-finite kernels $X_{j}$ appearing for $j\geq n$ for any $n \in \mathbb{N}$ it follows that $\pi^{-1}_{n}(\{x_{n}\})$ is not compact. Thus $\varprojlim X_{1} \times \cdots \times X_{n}$ is not locally compact.
	
	Let $\{n_{i}\}_{i=1}$ be subsequence indexing the countably many infinite kernels $X_{n_{i}} = \ker\big(f_{n_{i}-1,n_{i}}\big)$. Consider then the cofinal subsequence $(G_{n_{i}},f_{n_{i-1},n_{i}})$ of our original inverse system $(G_{n}, f_{n-1,n})$. Then by Proposition \ref{pr1}, we have that $$\varprojlim X_{1} \times 
	\cdots \times X_{n} \cong \varprojlim G_{n} \cong   \varprojlim G_{n_{i}} \cong \varprojlim Y_{1}\times \cdots \times Y_{i}$$ where we have that $Y_{i} := \ker(f_{n_{i-1},n_{i}}) \cong \mathbb{N}$. Thus we may enumerate each $Y_{i}$ as $Y_{i} = \{y^{k}_{i}\}_{k=1}^{\infty}$. Now define a map $$\Phi: \varprojlim Y_{1}\times \cdots \times Y_{i} \longrightarrow \mathbb{N}^\mathbb{N} \;\;\text{ sending }\;\; y = \{y^{k_{i}}_{i}\}_{i=1}\mapsto \{k_{i}\}_{i=1}.$$ Let $\phi_{i}$ be the map from $Y_{i}$ to the $i$th copy of $\mathbb{N}$ in $\mathbb{N}^\mathbb{N}$ sending $y^{k_{i}}_{i}$ to $k_{i}$.  A basis element for the topology on $\mathbb{N}^\mathbb{N}$ is of the form $$\bigcap_{j=1}^{k} p^{-1}_{j}(S_{j}),$$ where $p_{j}$ is the projection onto the $j$th coordinate. Then $$\Phi^{-1}\bigg(\bigcap_{j=1}^{k} p^{-1}_{j}(S_{j})\bigg) = \pi_{k}^{-1}\big(\phi^{-1}_{1}(S_{1}) \times \cdot \times \phi^{-1}_{k}(S_{k}) \big)$$ where $\phi^{-1}_{1}(S_{1}) \times \cdot \times \phi^{-1}_{k}(S_{k}) \subset Y_{1}\times \cdots \times Y_{k}$. This shows that both $\Phi$ and $\Phi^{-1}$ are continuous.
\end{proof}

We summarize our results on the following table. 

\begin{proposition}\label{invtop1}
	Let $(G_{\alpha}, f_{\alpha, \beta})_{\alpha, \beta \in \Lambda}$ be an inverse system of discrete abelian groups with $\Lambda$ admitting a countable cofinal subset. Then  $\varprojlim G_{\alpha}$ is homeomorphic to one of the following:
	
	\begin{table}[h]
		\centering
		\begin{tabular}{|l|c|c|c|}
			\hline
			& Stabilizes & Compactness & Homeomorphic to                              \\ \hline
			\hyperref[lma2b]{Case I.1}  & Yes    & compact    & Finite set                      \\ \hline
			\hyperref[lma3]{Case I.2}  & Yes & locally compact      & $\mathbb{N}$                       \\ \hline
			\hyperref[lma4]{Case II.1} & No   & compact      & $\mathcal{C}$                               \\ \hline
			\hyperref[lma5]{Case II.2} & No   & locally compact      &  $\mathbb{N} \times \mathcal{C}$  \\ \hline
			\hyperref[lma6]{Case II.3} & No  & not locally compact     &  $\mathbb{N}^{\mathbb{N}}$ \\ \hline
		\end{tabular}
	\end{table} where $\mathcal{C}$ denotes the Cantor set.
\end{proposition}

\vspace{5mm}

\section{Application to $KK$-Theory}

Finally, in this section we prove the main results of the paper by applying the results of section 5 along with the results of \cite{WilletYu25} stated below.

Let $A$ and $B$ denote separable $C^{*}$-algebras, $\pi: A \rightarrow \mathcal{L}(E)$ be a substantial graded representation on the Hilbert $B$-module $E$, where $(\pi,E) = (\pi_{0}\oplus\pi_{1}, H_{B}\oplus H_{B})$, with $\pi_{0} = \pi_{1}$ strongly absorbing representations. Here $H_{B}$ denotes the standard Hilbert $B$-module $H_{B} \cong \ell^{2}(\N) \otimes B$. We let $e \in \mathcal{L}(E)$ be the projection onto the first summand in $E$. See \cite[section 2]{WilletYu25} for relevant background.

\begin{definition}
	Let $X$ be a finite subset of the unit ball $A_{1}$ of $A$ and let $\varepsilon > 0$. We define $\mathcal{P}_{\varepsilon}^{\pi}(X,B)$ to be the set of self-adjoint contractions in $\mathcal{L}(E)$ such that \begin{itemize}
		\item[(i)] $p-e \in \mathcal{K}(E)$;
		\item[(ii)] $\|[p,a]\| < \varepsilon$ for all $a \in X$;
		\item[(iii)] $\|a(p^2 - p)\| < \varepsilon$ for all $a \in X$.
	\end{itemize} Equipping the norm topology on $\mathcal{P}_{\varepsilon}^{\pi}(X,B)$ that it inherits from $\mathcal{L}(E)$ then the controlled $KK$-theory groups is defined as \cite{WilletYu25} $$KK_{\varepsilon}^{\pi}(X,B) := \pi_{0}\big(\mathcal{P}_{\varepsilon}^{\pi}(X,B)\big).$$
\end{definition} Consider two isometries $s_{1},s_{2} \in \mathcal{B}(\ell^{2}) \subseteq \mathcal{L}(B)$ satisfying the Cuntz relation $s_{1}s_{1}^{*} + s_{2}s_{2}^{*} = 1$ and that commute with the subalgebra $A$ and $e$. We can define an operation on $KK_{\varepsilon}^{\pi}(X,B)$ by $$[p] + [q] := [s_{1}ps_{1}^{*} + s_{2}qs_{2}^{*}],$$ and indeed, this turns out to be a group operation by (Proposition 6.5, \cite{WilletYu25}). Note that this group operation gives $KK_{\varepsilon}^{\pi}(X,B)$ an abelian group structure.

\begin{definition}
	Let $\mathcal{X}$ be the set of all pairs $(X,\varepsilon)$ where $X$ is a finite subset of $A_{1}$ and $\varepsilon \in (0, \infty)$. There is partial order on $\mathcal{X}$ by defining $(X, \varepsilon) \leq (Y, \delta)$ if $\mathcal{P}_{\delta}^{\pi}(Y,B) \subseteq \mathcal{P}_{\varepsilon}^{\pi}(X,B)$.
\end{definition}

If $(X, \varepsilon) \leq (Y, \delta)$ then there is a canonical map $$\varphi_{(X,\varepsilon)}^{(Y,\delta)}: KK_{\delta}^{\pi}(Y,B) \rightarrow KK_{\varepsilon}^{\pi}(X,B),$$ defined as the canonical ``forget control" map. Equipping $KK_{\varepsilon}^{\pi}(X,B)$ with the discrete topology and group operation above then $\big(KK_{\varepsilon}^{\pi}(X,B)\big)_{(X,\varepsilon) \in \mathcal{X}}$ is an inverse system of topological groups. We may then define an inverse limit of topological groups as the following subspace of the product topology below: $$\varprojlim KK_{\varepsilon}^{\pi}(X,B) := \biggl\{ (x_{X,\varepsilon}) \in \prod_{(X,\varepsilon) \in \mathcal(X)} KK_{\varepsilon}^{\pi}(X,B)\; \bigg|\; \varphi_{(X,\varepsilon)}^{(Y,\delta)}(x_{Y, \delta}) = x_{X,\varepsilon} \biggr\}.$$ Let $\varphi_{(X,\varepsilon)}: \varprojlim KK_{\varepsilon}^{\pi}(X,B) \rightarrow KK_{\varepsilon}^{\pi}(X,B)$ denote the canonical map. Since the following set of all singletons $\big\{\{\,[p]\,\} \; | \; [p] \in KK_{\varepsilon}^{\pi}(X,B) \big\}$ is a basis for the topology on $KK_{\varepsilon}^{\pi}(X,B)$, then a basis for the inverse limit topology on $\varprojlim KK_{\varepsilon}^{\pi}(X,B)$ comprises sets of the form $\varphi_{(X,\varepsilon)}^{-1}(\{\,[p]\,\})$ where $(X, \varepsilon)$ runs through $\mathcal{X}$, see Remark \ref{basis intersection}.

Let $(a_{n})_{n=1}^{\infty}$ be a dense sequence in the unit ball $A_{1}$, and define $X_{n} := \{a_{1}, ... a_{n}\}$. Let $\varepsilon_{n} \rightarrow 0$ be an appropriate sequence tending to 0. Then the sequence $(X_{n}, \varepsilon_{n})_{n=1}^{\infty}$ is cofinal in $\mathcal{X}$ (Remark 6.7 \cite{WilletYu25}) and thus there is an isomorphism 
\begin{align*}
	\varprojlim KK_{\varepsilon}^{\pi}(X,B) \cong \varprojlim KK_{\varepsilon_{n}}^{\pi}(X_{n},B).
\end{align*} With this we may state one of the main results from $\cite{WilletYu25}$.

\begin{theorem}\label{WiYu1}
	For any pair of separable $C^{*}$-algebras $(A,B)$, the closure of zero in $KK(A,B)$ is isomorphic to $\varprojlim^1 KK_{\varepsilon_{n}}^{\pi}(X_{n},SB)$, and $KL(A,B) \cong \varprojlim KK_{\varepsilon_{n}}^{\pi}(X_{n},B)$ at topological groups where $\varprojlim^1 KK_{\varepsilon_{n}}^{\pi}(X_{n},SB)$ is endowed with the indiscrete topology. Moreover, there is a short exact sequence $$0 \rightarrow \lim_{\leftarrow}{}\!^1 KK_{\varepsilon_{n}}^{\pi}(X_{n},SB) \rightarrow KK(A,B) \rightarrow \lim_{\leftarrow} KK_{\varepsilon_{n}}^{\pi}(X_{n},B) \rightarrow 0.$$
\end{theorem}

The following is (\cite{WilletYu25}, Corollary 7.9).

\begin{theorem}\label{WiYu2}
	For any pair of separable $C^{*}$-algebras $(A,B)$, the closure of zero in $KK(A,B)$ is either $\{0\}$ or uncountable.
\end{theorem}

With Proposition \ref{invtop1} and Theorems \ref{WiYu1}, \ref{WiYu2}, we arrive at the main result of this section.

\newpage

\begin{theorem}\label{Topology}
	Let $A$ and $B$ be separable $C^{*}$-algebras. Then $KK(A,B)$ is homeomorphic to one of the following ten classes of topologies
	
	\begin{table}[h]
		\centering
		\begin{tabular}{|c|c|c|}
			\hline
			$KK_{\varepsilon_{n}}^{\pi}(X_{n},SB)$ is (ML) & $KK_{\varepsilon_{n}}^{\pi}(X_{n},B)$\; stabilizes & $KK(A,B)$ homeomorphic to                              \\ \hline
			Yes  & Yes        & $F$                   \\ \hline
			Yes  & Yes        & $\mathbb{N}$                       \\ \hline
			Yes & No         & $\mathcal{C}$                               \\ \hline
			Yes & No         &  $\mathbb{N} \times \mathcal{C}$  \\ \hline
			Yes & No         &  $\mathbb{N}^{\mathbb{N}}$ \\ \hline
			No  & Yes        & $F \times U$                      \\ \hline
			No  & Yes        & $\mathbb{N} \times U$                       \\ \hline
			No & No         & $\mathcal{C} \times U$                               \\ \hline
			No & No         &  $\mathbb{N} \times \mathcal{C} \times U$  \\ \hline
			No & No         &  $\mathbb{N}^{\mathbb{N}} \times U$ \\ \hline
		\end{tabular}
	\end{table} Where $F$ denotes some finite set and $U$ some uncountable set with the indiscrete topology.
\end{theorem}

\begin{proof}
	It follows from proposition \ref{ML} and theorem \ref{WiYu2} that $\varprojlim^1 KK_{\varepsilon_{n}}^{\pi}(X_{n},SB)$ vanishes if and only if the inverse system $KK_{\varepsilon_{n}}^{\pi}(X_{n},SB)$ is Mittag-Leffler. From Proposition \ref{top grp iso} and Theorem \ref{WiYu1}, it follows that $$KK(A,B) \cong \lim_{\leftarrow}{}\!^1 KK_{\varepsilon_{n}}^{\pi}(X_{n},SB) \times \lim_{\leftarrow} KK_{\varepsilon_{n}}^{\pi}(X_{n},B),$$ as topological groups.
\end{proof}

The following follows from corollary \ref{totdisc lemma} and answers question 7.1 from $\cite{Dad}$.

\begin{proposition}\label{Ror}
	The R{\o}rdam group, $KL(A,B) \cong \varprojlim KK_{\varepsilon_{n}}^{\pi}(X_{n},B)$, is totally disconnected.
\end{proposition}

\addresses

\end{document}